\documentclass[reqno, 12pt, reqno]{amsart}

\usepackage{amsfonts, amsthm, amsmath, amssymb, bbm, bm, color, enumerate,tikz-cd}
\usepackage{hyperref}
\usepackage[all]{xy} 
\usepackage[margin=1.5in]{geometry}

\usepackage{helvet}

\RequirePackage{mathrsfs} \let\mathcal\mathscr

\numberwithin{equation}{section}

\newtheorem{theorem}{Theorem}[section]

\newtheorem{lemma}[theorem]{Lemma}
\newtheorem{proposition}[theorem]{Proposition}

\theoremstyle{definition}

\newtheorem*{remark}{Remark}
\newtheorem{definition}[theorem]{Definition}

\renewcommand{\phi}{\varphi}
\renewcommand{\rho}{\varrho}

\newcommand{\PP}{\mathbb{P}}
\renewcommand{\AA}{\mathbb{A}}

\newcommand{\ZZ}{\mathbb{Z}}

\newcommand{\QQ}{\mathbb{Q}}

\renewcommand{\emptyset}{\varnothing}

\renewcommand{\leq}{\leqslant}
\renewcommand{\geq}{\geqslant}

\newcommand{\fp}{\mathfrak{p}}

\DeclareMathOperator{\Pic}{Pic}

\DeclareMathOperator{\Ker}{Ker}

\newcommand{\Br}{{\rm Br}}

\renewcommand{\t}{\mathbf{t}}

\title{Weak weak approximation for certain quadric surface bundles}
\author{Nick Rome}
\email{rome@tugraz.at}
\address{Graz University of Technology, Institute of Analysis and Number Theory,
Kopernikusgasse 24/II, 8010 Graz, Austria.}

\begin{document}
\maketitle

\begin{abstract}
We investigate weak approximation away from a finite set of places for a class of biquadratic fourfolds inside $\PP^2 \times \PP^3$, some of which appear in the recent work of Hassett--Pirutka--Tschinkel~\cite{HPT}. 
\end{abstract}

\section{Introduction}
Let $k$ be a number field of degree $d$ and let $\Omega$ denote the set of valuations on $k$. Given a smooth algebraic variety $X$ over $k$, we have the following embeddings
\[
X(k) \hookrightarrow X(\AA_{k}) \subset \prod_{\nu \in \Omega} X(k_{\nu}),
\]\vskip-1ex\noindent
where the first map is the diagonal embedding of the rational points into the ad{\'e}les. When the set of rational points on $X$ is non-empty, one would like to be able to discuss (either qualitatively or quantitatively) how the rational points on $X$ are distributed. We say that the variety $X$ satisfies \emph{weak approximation} if the image of $X(k)$ is dense in $X(\AA_{k})$, under the product topology. In his 1970 ICM talk, Manin observed that one can use the Brauer group, $\Br( X) = H^2_{\acute{e}t}(X, \mathbb{G}_m)$, to define a set $X(\AA_k)^{\Br}$ with $X(k) \subset X(\AA_k)^{\Br} \subset X(\AA_k)$, which can sometimes obstruct weak approximation. Colliot-Th{\'e}l{\`e}ne~\cite{CT} conjectured that this \emph{Brauer--Manin obstruction} is the only obstruction to weak approximation for any smooth, projective, geometrically integral and rationally connected variety. Our first result is to confirm Colliot-Th{\'e}l{\`e}ne's conjecture for a particular class of fourfolds over $k$.
\begin{theorem}\label{thm:BMOoo}
Let $k$ be a number field and $X/k$ the biprojective variety in $\PP^2 \times \PP^3$ defined by the equation 
\begin{equation}\label{eq:eqn}
xyt_1^2 + xzt_2^2 + yzt_3^2 + F(x,y,z)t_4^2 = 0,
\end{equation}
where $F$ is a non-degenerate ternary quadratic form over $k$.
Then the Brauer--Manin obstruction is the only obstruction to weak approximation for any smooth projective model of $X$.
\end{theorem}
There exist two general methods by which to prove theorems of this nature. The first is the descent method, a generalisation of the classical descent theory of elliptic curves via the use of universal torsors (c.f. ~\cite{Skorobook}). The second is the fibration method, in which one exploits the existence of a fibration $f:X \rightarrow Y$ where the base $Y$ and the fibres of $f$ satisfy the desired property. 
The varieties in Theorem \ref{thm:BMOoo} naturally admit the structure of quadric surface bundles over $\PP^2$ by the projection map 
\begin{align*}
X &\rightarrow \PP^2\\
(x:y:z;\t) &\mapsto (x:y:z),
\end{align*} and so it is the latter method which will be relevant.

The earliest example of the use of a fibration to study local-global principles in a family of varieties is due to Hasse in his proof of the local-global principle for quadratic forms in 4 variables (see e.g.~\cite[Ch. IV, Thm. 8]{Serre}). These ideas were then generalised by Colliot-Th{\'e}l{\`e}ne and Sansuc~\cite{CTSschinzel} who replaced Hasse's use of the prime number theorem in arithmetic progressions with an evocation of Schinzel's Hypothesis, which allowed them to show conditionally that the Brauer--Manin obstruction is the only one for a large class of conic bundles over $\PP^1$. This was extended by Colliot-Th{\'e}l{\`e}ne and Swinnerton-Dyer~\cite{CTSD}(building on work of Serre~\cite{SerreH}, Swinnerton-Dyer~\cite{SwD94} and Salberger~\cite{Sal88}) to establishing the conjecture (conditionally on a variant of Schinzel's Hypothesis) for pencils of \emph{generalised} Severi--Brauer varieties (in the language of \cite{CTSD}), of which quadric surfaces are an example. Skorobogatov~\cite{Skoro}(and subsequently Colliot-Th{\'e}l{\`e}ne--Skorobogatov~\cite{CTS}) was able to establish unconditionally that the Brauer--Manin obstruction is the only one for quadric surface bundles over $\PP^1$ of rank $\leq 3$, by combining the fibration and descent methods. Finally, if all of the degenerate fibres of a quadric surface bundle over $\PP^1$ are defined over $\QQ$, then the Brauer--Manin obstruction is unconditionally known to be the only one thanks to the work of Browning--Mattheisen--Skorobogatov~\cite[Theorem 1.4]{BMS}.

Over higher dimensional bases, much less is known. Conditional on Schinzel's Hypothesis, Wittenberg~\cite[Corollaire 3.6]{Wittenberg} has shown that a fibration over $\PP^n$ into generalised Severi--Brauer varieties has the property that the Brauer--Manin obstruction is the only one. 
Theorem \ref{thm:BMOoo} gives a rare example of an unconditional proof that the Brauer--Manin obstruction is the only one for a fibration over a higher dimensional base.
We take this moment to point out that conic bundle fibrations have a vast and illustrious history in the literature and that fibrations over projective space into quadrics of dimension $\geq 3$ always satisfy weak approximation.

Our motivation for studying the particular class of varieties that we do comes from a recent breakthrough by Hassett--Pirutka--Tschinkel \cite{HPT} on classical rationality problems in algebraic geometry. %
They were able to show for the first time that rationality is not deformation invariant in families of complex fourfolds.
A key part of the proof was to show that there exists a non-trivial element in the (complex) Brauer group of the variety defined by equation \eqref{eq:eqn}
with \begin{equation}\label{eq:FHPT}
F(x,y,z) = x^2 + y^2 + z^2 - 2(xy+xz+yz).\end{equation}
The major novelty of this paper is to use ideas inspired by Hassett--Pirutka--Tschinkel (and others in the area) over \emph{non algebraically closed} fields in order to apply them to the study of problems of a Diophantine nature.

Varieties defined by equations of the form \eqref{eq:eqn} may fail weak approximation (as discussed at the end of Section \ref{sec:invs}). We recall the following slightly weaker notion regarding the qualitative distribution of rational points.
\begin{definition}
Let $X$ be a smooth projective variety over a number field $k$ and let $S$ be a finite set of places of $k$. We say that $X$ satisfies weak approximation away from $S$ if the image of the diagonal map from $X(k)$ to $\prod_{\nu \in \Omega \setminus S} X(k_{\nu})$ is dense.
\end{definition}
A variety $X$ satisfying this property for a certain finite set $S$ is said to satisfy \emph{weak weak approximation}, a property with close connections to the inverse Galois problem \cite[$\S$3.5]{SerreGalois}. A number of recent papers have studied weak weak approximation including (but not limited to) for del Pezzo surfaces \cite{wwasam, wwadp2}, cubic hypersurfaces and double elliptic surfaces \cite{wwanonrational,wwajulian,wwarankjump} and even for certain Campana orbifolds \cite{wwacampana}. 

The main result of this paper is to establish weak weak approximation for the varieties appearing in Theorem \ref{thm:BMOoo}. 
\begin{theorem}\label{thm:WAthm}
Let $k$ be a number field and $X/k$ the biprojective variety in $\PP^3 \times \PP^2$ defined by the equation 
\[
xyt_1^2 + xzt_2^2 + yzt_3^2 + F(x,y,z)t_4^2 = 0,
\] 
where $F$ is a non-degenerate ternary quadratic form over $k$ 
such that \begin{equation}\label{eq:Fcond} \begin{cases}\text{The forms } F(0,y,z) , F(x,0,z), \text{ and } F(x,y,0) \text{ are all squares},\\
F(x,y,z) \not \in k(x,y,z)^2. \end{cases}
\end{equation} 
Let $S \subset \Omega_k$ denote the set of archimedean places and places above the rational prime 2. Then $X$ satisfies weak approximation away from $S$.
Moreover, if $F$ takes only positive values in all real embeddings, then we may reduce $S$ to just the primes above 2.
\end{theorem}

\begin{remark}
That weak approximation is controlled by the archimedean places and places above 2 is a consequence of Proposition \ref{prop:locconst}, which states that the Brauer--Manin obstruction is trivial for these varieties at non-archimedean places above odd primes.
\end{remark}

It is possible to construct examples where weak approximation does indeed fail at one of the places in $S$ (c.f.\ the final remark of the paper). 
We note that this provides new families of examples of failure of weak approximation on rational varieties caused by a transcendental Brauer group element. This can never occur for the more familiar setting of pencils of conics but Harari has previously provided examples given by conic bundles over higher dimensional bases~\cite{HarTrans}. 
\begin{remark}
The first condition is a little odd looking from a geometric point of view. Perhaps more natural would be the condition that the conic $F(x,yz) =0$ is tangent to each of the coordinate axis. However, as we shall see in the course of the proof, this condition is not strong enough for fields $k$ which are not algebraically closed. Tangency would allow us to deduce that  $F(0,y,z)$ is square in $k[y,z]$ modulo constants, however we really need for it to be a square.
\end{remark}
\begin{remark}
The Hassett--Pirutka--Tschinkel example \eqref{eq:FHPT} certainly satisfies condition \eqref{eq:Fcond}.
\end{remark}

\begin{remark}
Our varieties are not smooth and hence we must clarify what we mean by weak approximation in this setting. Following Colliot-Th{\'e}l{\`e}ne--Xiu~\cite[Section 8]{CTX}, we say that a singular variety $V$ satisfies weak approximation if for any finite set of places of $k$, there exists a resolution of singularites $\widetilde{V} \xrightarrow{\phi} V$ such that the $k$ points of the smooth locus of $V$ are dense in the set $\prod_{\nu \in S} \phi(\widetilde{V}(k_{\nu}))$. Note that this definition is independent of the resolution of singularites chosen.
\end{remark}

The second scenario in Theorem \ref{thm:WAthm} where the rational points are restricted to a connected component occurs often when $\Br(X)/\Br(k)$ is finite, as is the case here, see for instance \cite[\S 3]{SwD62} or  \cite[Prop 7.2]{CTCS}.

\subsection*{Acknowledgements}
This paper formed part of the author's thesis and he is grateful to his supervisor Tim Browning for suggesting the problem and his continued guidance. The author has greatly benefited from a number of useful discussions with Julian Lyczak to whom he is indebted. The author would also like to thank Christopher Frei, Tim Dokchitser, Wei Ho and, in particular, Diego Izquierdo for their valuable comments on previous drafts. The author is grateful for the suggestions of anonymous referees which have significantly reshaped the exposition of the paper. The author is funded by FWF project ESP 441-NBL.

\section{Proof of Theorem \ref{thm:BMOoo}}\label{sec:BMOproof}
As discussed in the introduction, we will use the fibration method to establish Theorem \ref{thm:BMOoo}, specifically Harari's fibration method with a smooth section.
\begin{theorem}[{\cite[Th{\'e}or{\`e}me 4.3.1]{HarariDuke}}]\label{thm:fibmethod}
Let $V$ and $B$ be geometrically integral varieties over a field $k$ such that $B$ satisfies weak approximation and there exists a dominant morphism $V \xrightarrow{\pi} B$ which admits a section $s$.
The Brauer--Manin obstruction is the only one for any smooth projective of model of $V$ if the following are satisfied:
\begin{enumerate}
\item The generic fibre $V_{\eta}$ is a geometrically integral variety over $k(B)$ and $s$ defines a smooth point in $V_{\eta}$.
\item For any smooth projective model $W$ of $V_{\eta}$, $\Br(W_{\overline{k(B)}}) := \Br(W \times_{k(B)} \overline{k(B)})$ is trivial and $\Pic(W_{\overline{k(B)}})$ has no torsion.
\item There exists a non-empty open $U \subset B$ such that $\forall b \in U$ the Brauer--Manin obstruction is the only one for all smooth proper models of $V_{b}$.
\end{enumerate}
\end{theorem}
One would like to apply this theorem to the natural fibration of $X$ over $\PP^2$ however one cannot guarantee the existence of a smooth section. To sidestep this issue, we consider a quasiprojective variety $V$ obtained by dehomogenising the equation defining $X$. Specifically, if we set $z=1$, the resulting equation 
\begin{equation}\label{eq:dehom}
xyt_1^2 +xt_2^2 + y t_3^2 + F(x,y,1)t_4^2 = 0,
\end{equation}
defines the variety $V$ inside $\AA^2_{x,y} \times \PP^3$.  Let $\pi$ denote the map to $\AA^1_x$ that projects onto the $x$ variable. The fibres of $\pi$ are threefolds in $\AA^1_y \times \PP^3$ which admit the structure of quadric surface bundles over $\AA^1_y$.  Therefore, a proper smooth model of a fibre of $\pi$ above a closed point on $\AA^1_x$ will be a quadric surface bundle over $\PP^1$.

A systematic study of the arithmetic of quadric surface bundles over $\PP^1$ was conducted by Skorobogatov~\cite{Skoro}. We recall here an important result from that investigation.
\begin{lemma}[{\cite[Corollary 4.1]{Skoro}}]\label{lem:qsbbmo}
Let $Z/k$ be a quadric surface bundle over $\PP_k^1$ such that the fibre above at most two closed geometric fibres is defined by a quadratic form of rank $\leq 2$. Then, for any smooth proper model of $Z$, the Brauer--Manin obstruction is the only obstruction to weak approximation and to the existence of rational points.
\end{lemma}

We will now use this information to deduce the necessary facts about the fibration over $\AA^1_x$ constructed above in order to apply Harari's theorem.

\begin{proposition}\label{prop:brconds}
Let $V$ be the quasiprojective variety defined by equation \eqref{eq:dehom} and let $\pi$ be the projection $V \rightarrow \AA^1_x$ onto the $x$ variable. Let $V_{\eta}$ denote the generic fibre of the map $\pi$. Then
\begin{enumerate}
\item The variety $V_{\eta}$ is a geometrically integral variety over $k(x)$ which admits a smooth $k(x)$ point.
\item For any smooth projective model $W$ of $V_{\eta}$, the geometric Picard group $\Pic(W_{\overline{k(x)}})$ has no torsion.
\item For any smooth projective model $W$ of $V_{\eta}$, the geometric Brauer group $\Br(W_{\overline{k(x)}})$ is trivial.
\item There exists a non-empty open $U \subset \AA^1$ such that for any $x \in U$, the Brauer--Manin obstruction is the only one for all smooth proper models of the fibre of $\pi$ above $x$.
\end{enumerate}
\end{proposition}

\begin{proof}
\begin{enumerate}
\item The generic fibre of $\pi$, which we have denoted $V_{\eta}$, is a quadric surface bundle in $k(x)$ over $\AA^1_y$ with one degenerate fibre (above $y=0$) and thus is geometrically integral. Moreover, $\pi$ admits a section given by
\[
x \mapsto (x,0;0,0,1,0),
\] which is easily checked to be generically smooth.
\item In fact, the generic fibre must be geometrically rational and thus so must $W$ be (the proof given here mirrors \cite[Thm 3.3]{CTSSD}). Indeed, after base change to $\overline{k(x)}$, the function field of $V_\eta$ is the function field of a quadric surface over $\overline{k(x)}(\PP^1)$. Tsen's theorem implies this quadric has a $\overline{k(x)}(\PP^1)$ point and hence the functional field is purely trancendental over $\overline{k(x)}(\PP^1)$ and thus over $\overline{k(x)}$. Since $W$ is geometrically rational, it is rationally connected and hence $\Pic(W_{\overline{k(x)}})$ is torsion free (e.g.\ \cite[Cor 4.4.4]{Brauerbook}).
\item Similarly, because $W$ is geometrically rational, and since the Brauer group is a birational invariant of smooth varieties, $\Br(W_{\overline{k(x)}})$ is trivial.
\item Consider the set $U=\{x \neq 0\} \subset \AA^1_x$. The fibre of $\pi$ above any closed point $x$ in this set is a quadric surface bundle threefold over $\AA^1_y$. 
If $y\neq0$ then the resulting quadric surface is defined by the vanishing of a quadratic form in $k(x)$ of rank at least 3.
A smooth proper model of $\pi^{-1}(x)$ is therefore a quadric surface bundle over $\PP^1$ with at most 2 \emph{essentially singular} fibres above $y=0$ and $y=\infty$, that is fibres defined by the vanishing of a quadratic form of rank at most 2. Hence by Lemma \ref{lem:qsbbmo}, the Brauer--Manin obstruction is the only one.
\end{enumerate}
\end{proof}

Let $\widetilde{X}$ denote any smooth projective model of the variety $X$ in Theorem \ref{thm:BMOoo}. Then  $\widetilde{X}$ also represents a smooth projective model of the quasiprojective variety $V$ defined by equation \eqref{eq:dehom}. Combining Theorem \ref{thm:fibmethod} and Proposition \ref{prop:brconds}, we deduce Theorem \ref{thm:BMOoo}.

\section{Computing the Brauer Group}\label{sec:unramcoh}
Throughout this and all subsequent sections we fix a number field $k$, let $X/k$ be the variety in the statement of Theorem \ref{thm:WAthm} and $\widetilde{X}$ a fixed desingularisation. %
Our aim in this section is to establish the following description of the Brauer group of $\widetilde{X}$.
\begin{theorem}\label{thm:brgrp}
We have $$\Br (\widetilde{X}) / \Br (k) \cong \ZZ/2\ZZ$$ and the quotient is generated by the image of the class $(-xz,-yz)_{k(\PP^2)}$ under the map $\Br (k(\PP^2)) \rightarrow \Br (\widetilde{X})$.
\end{theorem}

Our strategy is inspired by \cite[$\S$ 3.6]{Pirutka} (although greatly streamlined with the help of the anonymous referee). In particular, we will use the unramified cohomological description of the Brauer group.
\begin{proposition}[{\cite[Prop 3.7]{Pirutka}}]\label{prop:unramtobr}
If $V$ is a smooth projective variety over $k$ then
\[
H^2_{nr}(k(V)/k, \ZZ/2\ZZ) \simeq \Br(V)[2].
\]
\end{proposition}
The unramified cohomology is computed using residues
\[
H^2_{nr}(k(V)/k, \ZZ/2\ZZ) = \bigcap_{\nu} \text{Ker}\left(H^2(k(V), \ZZ/2\ZZ) \xrightarrow{\partial_\nu^2} H^1(\kappa(\nu), \ZZ/2\ZZ)\right),
\]
where the intersection is taken over all discrete valuations $\nu$ of rank one on $k(X)$ which are trivial on $k$.
Note that the purity theorem allows us, for smooth varieties $V$, to write the unramified cohomology in terms of codimension one points 
\[
H^2_{nr}(k(V)/k, \ZZ/2\ZZ) = \bigcap_{x \in V^{(1)}} \text{Ker}\left(H^2(k(V), \ZZ/2\ZZ) \xrightarrow{\partial_x} H^1(\kappa(x), \ZZ/2\ZZ)\right),
\]
with the intersection running over codimension one points $x$ where $\partial_x$ is the associated residue map and $\kappa(x)$ the residue field.
Since the unramified cohmology only depends on the function field, we have
\[
H^2_{nr}(k(X)/k, \ZZ/2\ZZ) \simeq H^2_{nr}(k(\widetilde{X})/k, \ZZ/2\ZZ) \simeq \Br(\widetilde{X})[2]
\] 
This means in particular that we need not explicitly construct the desingularisation $\widetilde{X}$ in order to understand its Brauer group.

\subsection{Preliminaries}
Let $X_{\eta}$ denote the generic fibre of the map $X \rightarrow \PP^2$ which is a quadric surface over the function field $k(\PP^2)$. 
We will need the following classical results on the cohomology of quadrics.
\begin{proposition}[{\cite[Prop. 6.2.3]{Brauerbook}}]\label{prop:quadbr}
Let $K$ be a field with $\text{char}(K) \neq 2$ and $Q/K$ a smooth projective quadric of dimension 1 or 2. Then the map $\Br(K) \rightarrow \Br(Q)$ is surjective. Moreover, 
\begin{enumerate}
\item Suppose Q is a conic. If $Q(K) \neq \emptyset$ then the map is an isomorphism. If $Q(K) = \emptyset$, then the kernel of the map is isomorphic to $\ZZ/2\ZZ$.
\item Suppose $Q$ is a quadric surface. If the discriminant of $Q$ is a non-square in $K$ then the map is an isomorphism. If the discriminant is square then the map is an isomorphism if $X(K) \neq \emptyset$. If the discriminant is square and $X(K) = \emptyset$, then the kernel is isomorphic to $\ZZ/2\ZZ$.
\end{enumerate}
\end{proposition}

\begin{lemma}[{\cite[Thm 3.10]{Pirutka}}]\label{lem:H1ref}
Let $Q$ be a quadric defined by the vanishing of a non-degenerate quadratic form $q$ over $K$. 
\begin{enumerate}
\item If $q$ has rank at least 3, then the natural map
\[
H^1(K, \ZZ/2\ZZ) \rightarrow H^1_{\text{nr}}(K(Q)/K, \ZZ/2\ZZ)
\] is injective.
\item If $q$ has rank 2, then the map 
\[
H^1(K, \ZZ/2\ZZ) \rightarrow H^1_{\text{nr}}(K(Q)/K, \ZZ/2\ZZ)
\] has kernel generated by the class of the discriminant of $q$. In particular, the kernel is trivial if the discriminant is square and $\ZZ/2\ZZ$ if the discriminant is non-square.
\end{enumerate}
\end{lemma}

\subsection{Non-trivial Brauer classes}
Let $K = k(\PP^2)$.
We begin by noting that
\[
H^2_{\text{nr}}(k(X)/k, \ZZ/2\ZZ) \hookrightarrow H^2_{\text{nr}}(k(X_\eta)/k, \ZZ/2\ZZ) \subset H^2_{\text{nr}}(K(X_\eta)/K,\ZZ/2\ZZ).
\]
Thus any element in $\Br(\widetilde{X})[2]$ can be mapped to an element in $H^2(K(Q)/K, \ZZ/2\ZZ)$ where $Q$ is the quadric defining $X_\eta$ over $K$. This can then be understood using the lemmata of the previous subsection.

\begin{lemma}\label{lem:finale}
The pullback of the class $(-xz,-yz)_{k(\PP^2)}$ to $k(\widetilde{X})$ is unramified and thus lies in $\Br \widetilde{X}$.
Moreover, for any $\gamma \in \Br( k(\PP^2))$ and $D \in (\PP^2)^{(1)}$ such that $\gamma$ gives rise to a nonconstant element in $\Br \widetilde{X}$ and $\gamma$ is ramified along $D$, then $D$ must be a component of $\{xyz=0\}$ and $\partial_D(\gamma) = \partial_D((-xz,-yz)_{k(\PP^2)})$.
\end{lemma}

\begin{proof}
Let $q$ be the quadratic form in $K = k(\PP^2)$ defining the generic fibre $X_\eta$.
Let $D \in \left(\PP^2\right)^{(1)}$ be a codimension one point and consider $A$ the associated local ring $\mathcal{O}_{\PP^2, D}$. Without loss of generality, $q$ has coefficients in $A$ and $q=0$ defines a closed subscheme $Z$ in $\PP^4_A$. Let $\pi$ be a uniformiser for $A$, then the divisor $\pi=0$ on $Z$ is either integral or a union of two planes. In either case, we can choose a discrete valuation $\nu$ on $K(X_\eta)$ associated to it. If $\pi = 0$ is integral then it is a discrete valuation and otherwise we take $\nu$ corresponding to one of the two planes. Either way, we have the following commutative diagram from \cite[Prop 3.4]{Pirutka} (with $e=1$)
\begin{center}
\begin{tikzcd}
H^2(K(X_\eta), \ZZ/2\ZZ) \arrow[r, ] & H^{1}(\kappa(\nu), \ZZ/2\ZZ) \\
H^2(K, \ZZ/2\ZZ) \arrow[r, ] \arrow[u, "Res_{K/K(X\eta)}"] & H^{1}(\kappa(D), \ZZ/2\ZZ) \arrow[u,"Res_{\kappa(D)/\kappa(\nu)}"].
\end{tikzcd}
\end{center}
The commutativity of the diagram implies that $\partial_x(\gamma) \in \Ker(\text{Res}_{\kappa(D)/\kappa(\nu)})$ for any $\gamma \in H^2(K,\ZZ/2\ZZ)$. We will exploit this fact to understand which classes in $\Br(K)$ can land in $\Br(\widetilde{X})$.

Firstly, we claim that at least two of the coefficients of $q$ must have odd valuation with respect to $D$.
If not then the reduction of $q=0$ modulo $D$ must be either a smooth conic or smooth quadric surface. Therefore, by Lemma \ref{lem:H1ref}(1), the map  
$\text{Res}_{\kappa(D)/\kappa(\nu)}$ is injective. Moreover, since the image of $\gamma$ in $\Br(k(X))$ lies in $\Br(\widetilde{X})$ it must be unramified along the valuation $\nu$. Hence, by the commutativity of the diagram, $\gamma$ is unramified at $D$.

This means that at least two of the coefficients of $q$ must have odd valuation with respect to $D$ and thus $D$ can only be one of the components of $\{xyz=0\}$. The reduction of $q=0$ modulo $D$ is no longer smooth, but after passing to a quadratic extension $F_D/k(D)$ it decomposes as a union of two transveral lines. Therefore if $\gamma$ gives rise to an element in $\Br(\widetilde{X})$, we must have $\gamma \in \Ker[(H^1(k(D), \ZZ/2\ZZ) \rightarrow H^1(F_D, \ZZ/2\ZZ)]$.
Suppose $D$ corresponds to the line $\{x=0\}$. Then, $k(X_\nu)$ is the function field of the quadric over $k(D)$ defined by the equation 
\[
\overline{yz}U^2 + \overline{F(x,y,z)}V^2 = 0.
\]
By Lemma \ref{lem:H1ref} (2), $\Ker \text{Res}_{D/\nu}$ is generated by the class of the discriminant of $\overline{q}$ which is $-\overline{yzF(x,y,z)} \in \kappa(D)^*/(\kappa(D)^*)^2$. However, $\overline{F(x,y,z)}$ is  in $(\kappa(D)^*)^2$, by conditions \eqref{eq:Fcond}. Thus the class is given by $-\overline{yz}$, which coincides with the residue of the class $(-xz,-yz)$. By symmetry, the exact same calculation holds for the other two lines.
\end{proof}

We are now ready to prove the main result of this section.

\begin{proof}[Proof of Theorem \ref{thm:brgrp}]
The previous lemma proves that the pullback of the class $\beta : = (-xz,-yz)_{k(\PP^2)}$ to $k(\widetilde{X})$ lands in $\Br (\widetilde{X})$. Since $F(x,y,z)$ is non-square in $k(x,y,z)$ by assumption and since $\beta$ has non-trivial residues, we conclude by Proposition \ref{prop:quadbr}(2) that the pullback to $k(\widetilde{X})$ is also non-trivial and non-constant.
Moreover, Lemma \ref{lem:finale} states that any non-trivial element in $\Br (\widetilde{X})$ must ramify along the coordinate axes and have the same residues as $\beta$.
Suppose that $\gamma$ is such a class. Since $\Br (\AA^2)$ is trivial, it cannot be the case that $\gamma$ only ramifies along one axis. Moreover, $\gamma$ cannot ramify along just two axes because then $\gamma - \beta$ would only be ramified along one axis. Therefore $\gamma$ must ramify along all three coordinate axes and hence $\partial_D(\gamma) = \partial_D(\beta)$ for all codimension one points. 
Hence the Brauer group is generated by $\beta$ and $\Br (k)$. 
\end{proof}

\begin{remark}
In this section, we have made pivotal use of the conditions \eqref{eq:Fcond}. %. 
One intuitive justification for the theorem is given by Abhyankar's Lemma (as observed by Colliot-Th{\'e}l{\`e}ne~\cite[\S 3]{CT2}). Namely, since the ramification locus of the quaternion algebra $(-xy,-yz)_{k(\PP^2)}$ is contained within the ramification locus of $X \rightarrow \PP^2$, the ramification cancels in $\widetilde{X}$ (see e.g.\ \cite[p. 116]{MordellWeil} for an example of this phenomenon). Concretely, in this case $-yz$ is not a square modulo $x$ in $k(\PP^2)$ but \emph{is} a square modulo $x$ in $k(X)$. This explains why the quaternion algebra $(-xz,-yz)$ has nontrivial residues but these residues vanish when you move from $\PP^2$ to $\widetilde{X}$.
\end{remark}

\section{The Brauer--Manin Obstruction set}\label{sec:invs} 
In this final section, we explicitly compute the Brauer--Manin obstruction, which combined with Theorem \ref{thm:BMOoo} finishes the proof of Theorem \ref{thm:WAthm}. Recall the definition of the Brauer--Manin obstruction set (e.g.\ \cite[Def. 8.2.5]{Poonen})
\[
X(\AA_k)^{\Br} = \left\{ (P_{\nu})_{\nu} \in X(\AA_k) : \sum_{\nu} \text{inv}_{\nu} \text{ev}_{\gamma}(P) = 0 \in \QQ/\ZZ\,\, \forall \gamma \in \Br(X)\right\}.
\]
Theorem \ref{thm:WAthm} follows immediately from explicitly determining the image of the evaluation maps.
\begin{proposition}\label{prop:locconst}
Let $\beta \in \Br(X)$ be the non-trivial class $(-x/z,-y/z)$.
Let $\nu$ be a non-archimedean place lying above an odd rational prime. Then $\text{inv}_{\nu} \text{ev}_{\beta}$ is identically zero on all $k_{\nu}$ points.
\end{proposition}
Since Theorem \ref{thm:BMOoo} established that the Brauer--Manin obstruction is the only obstruction to weak approximation, and since $\beta$ is the only non-trivial Brauer class (up to elements of $\Br(k)$, which have no impact on the obstruction), this proposition shows that the failure of weak approximation is completely determined by the archimedean places and those places above 2.

\begin{proof}
Proposition \ref{prop:locconst} will be proved by an explicit computation of the invariant maps. Throughout, $(a,b)_{\nu}$ will refer to the Hilbert symbol associated to the local field $k_{\nu}$. Note that by the continuity of the Brauer--Manin pairing, we may assume that all of the coordinates $x_\nu, y_\nu$ and $z_\nu$ are non-zero.
The evaluation of the class $\beta$ at a point $P \in X(\AA_k)$ is equal to the evaluation of the pre-image $\alpha \in \Br(k(\PP^2))$ at the point $B \in \PP^2(\AA_k)$ below $P$. Fix a point $P = (x_{\nu}:y_{\nu}:z_{\nu};\t_{\nu})_{\nu} \in X(\AA_k)$, then
\[
\text{inv}_{\nu} \text{ev}_{\beta} (x_{\nu}:y_{\nu}:z_{\nu};\t_{\nu}) = \frac{1}{2} \text{ if and only if } (-x_{\nu}y_{\nu}, -y_{\nu}z_{\nu})_{\nu} = -1.
\]
Let $\fp$ denote the prime ideal of $k$ which corresponds to the place $\nu$. We will break the proof into two cases: 
\begin{enumerate}
\item $x_{\nu}y_{\nu}z_{\nu}$ is nonzero modulo $\fp$, 
\item  $x_{\nu}y_{\nu}z_{\nu}$ is 0 modulo $\fp$.
\end{enumerate}
In the first case, the associated Hilbert symbol is +1. In the second case, we may assume, without loss of generality, that at least one of the coordinates $x_{\nu},y_{\nu}$ and $z_{\nu}$ is non-zero modulo $\fp$.
Suppose that $x_{\nu}$ is 0 modulo $\fp$ and that $y_{\nu}$ and $z_{\nu}$ are non-zero modulo $\fp$. We know that there exist $t_{i,\nu} \in k_{\nu}$ for $i =1, \ldots,4$ such that
\begin{equation}\label{eq:vadic}
x_{\nu}y_{\nu}t_{1,\nu}^2 + x_{\nu}z_{\nu}t_{2,\nu}^2 + y_{\nu}z_{\nu}t_{3,\nu}^2 + F(x_{\nu},y_{\nu},z_{\nu})t_{4,\nu}^2 = 0.
\end{equation}
Reducing modulo $\fp$, we get 
\[
y_{\nu}z_{\nu}t_{3,\nu}^2 + F(0,y_{\nu},z_{\nu})t_{4,\nu}^2 = 0.
\]
If $F(x_{\nu}, y_{\nu}, z_{\nu}) \not \equiv 0$ mod $\fp$ then,
by assumption, $F(x_{\nu}, y_{\nu}, z_{\nu})$ is square modulo $\fp$. Therefore we deduce that $-y_{\nu}z_{\nu}$ is as well, thus $(-x_{\nu}z_{\nu}, - y_{\nu} z_{\nu})_{\nu} = +1$. 
If $F(x_{\nu}, y_{\nu}, z_{\nu}) = 0$ in $k_\nu$ then the equation \eqref{eq:vadic} becomes
\[
x_{\nu}y_{\nu}t_{1,\nu}^2 + x_{\nu}z_{\nu}t_{2,\nu}^2 + y_{\nu}z_{\nu}t_{3,\nu}^2 = 0.
\]
Multiplying by $x_{\nu}y_{\nu}$ this is precisely the equation whose solubility the Hilbert symbol $(-x_\nu z_\nu, - y_\nu z_\nu)_\nu$ detects. Finally, suppose that $F(x_{\nu}, y_{\nu}, z_{\nu}) \neq 0$ in $k_\nu$ but $F(x_{\nu}, y_{\nu}, z_{\nu}) \equiv 0$ mod $\fp$. The solubility of \eqref{eq:vadic} is equivalent to the solubility of the equation
\begin{equation}\label{modded}
z_\nu A^2 + y_\nu B^2 + x_\nu C^2 + x_\nu y_\nu z_\nu F(x_\nu, y_\nu, z_\nu) D^2 = 0.
\end{equation}
Reducing modulo $\fp^2$, we arrive at 
\[
z_\nu A^2 + y_\nu B^2 + x_\nu C^2 \equiv 0,
\] a conic which is soluble exactly when $(-x_{\nu}z_\nu, - y_\nu z_\nu)_\nu = +1$. Since there must be a solution of \eqref{modded} where not all of $A,B$ and $C$ are 0 mod $\fp$, we have a mod $\fp$ solution to the conic.
If instead $y_{\nu}$ or $z_{\nu}$ are 0 modulo $\fp$ then the proof is similar.

Finally, we address the situation where two coordinates, say $x_\nu$ and $y_\nu$, have non-zero valuation. By assumption, the value $F(0,0,z)$ of the quadratic form is in $(k^\times)^2$ for any $z \in k^\times$ and thus 
$F(x_\nu,y_\nu,z_\nu)$ is square modulo $\fp$. By the assumption \eqref{eq:Fcond}, it follows that the quadratic form $F$ has to be of the shape
\[
F(x,y,z)
=
a^2x^2 + b^2y^2 +c^2z^2 \pm 2(abxy + acxz + bcyz),
\] for some $a,b,c \in k^\times$. Therefore, we have
\[
\nabla F(x_\nu,y_\nu,z_\nu)
=
(2a^2x \pm 2(aby + acz), 2b^2y \pm 2(abx + bcz), 2c^2z \pm 2(acx + bc y))
.
\] 
The only way for this to be congruent to $(0,0,0)$ modulo $\fp$ is if $c \equiv 0 $ modulo $\fp$, in which case $F(x_\nu, y_\nu, z_\nu)$ vanishes modulo $\fp$. If this is the case then we proceed as in the previous situation, reducing \eqref{modded} by $\fp$ raised to the power of the valuation of $x_\nu y_\nu F(x_\nu, y_\nu, z_\nu)$. Otherwise, we may apply Hensel's lemma to deduce that $F(x_\nu, y_\nu, z_\nu)$ is a square in $k_\nu^\times$.
As a result, $A := y_\nu z_\nu t_{3,\nu}^2 + F(x_\nu,y_\nu,z_\nu)t_{4,\nu}^2$ is the norm of an element in $k_\nu(\sqrt{-y_\nu z_\nu})$. 
Moreover,
\[
(-x_\nu z_\nu)\times A 
=
(-x_\nu z_\nu) \times (-x_\nu y_\nu t_{1, \nu}^2 - x_\nu z_\nu t_{2,\nu}^2)
=
y_\nu z_\nu (x_\nu t_{1,\nu})^2 +
(x_\nu z_\nu t_{2, \nu})^2 
\] is also. Therefore $-x_\nu z_\nu$ must be the norm of an element in $k_\nu(\sqrt{-y_\nu z_\nu})$ and thus $(-x_\nu y_\nu, -z_\nu z_\nu)_\nu = +1$.
\end{proof}

\begin{remark}
Informally speaking, the reason this result is true is similar to the reason why $\beta$ lives in $\Br(X)$. Namely, if a prime divides $x$, then the nature of the equation forces $-yz$ to be a square, which kills the invariant map.
\end{remark}

Finally, we will study the invariant map at archimedean places of $k$. If $\nu$ is a complex place then the conic $x_{\nu}z_{\nu}U^2 + y_{\nu}z_{\nu} V^2 + W^2 = 0$ always has points and thus the invariant map is identically zero. The value of the invariant map at real places is determined by the signs of $x,y$ and $z$.

\begin{proposition}\label{prop:arch}
Let $\nu$ be a real place of $k$ and let $P_{\nu}=(x_{\nu}:y_{\nu}:z_{\nu};\t_{\nu}) \in X(k_{\nu})$. 
If $F(x_{\nu},y_{\nu},z_{\nu})>0$ then $\text{inv}_{\nu} \text{ev}_{\beta} (P_{\nu}) = +1$.
Otherwise, $$\text{inv}_{\nu} \text{ev}_{\beta} (P_{\nu}) = \begin{cases} \frac{1}{2} \quad{} \text{ if } x,y,z \text{ all have the same sign in } k_{\nu},\\ 0 \quad{} \text{ otherwise.} \end{cases}$$
\end{proposition}
\begin{proof}
The explicit description of the real Hilbert symbol is
\[
(-x_{\nu}z_{\nu},-y_{\nu}z_{\nu})_{\nu} = \begin{cases} - 1 \text{ if } -x_{\nu}z_{\nu}<0 \text{ and } - y_{\nu}z_{\nu} < 0, \\ +1 \text{ otherwise.} \end{cases}
\]
If $F(x_{\nu},y_{\nu},z_{\nu})>0$ then the Hilbert symbol must be +1. Indeed, in this case, at least one of $x_{\nu}y_{\nu}, x_{\nu}z_{\nu}$ or $y_{\nu}z_{\nu}$ must be negative. If $x_{\nu}y_{\nu}$ is negative then $x_{\nu}$ and $y_{\nu}$ have differing signs and thus at least one of $x_{\nu}z_{\nu}$ or $y_{\nu}z_{\nu}$ must also be negative.
\end{proof}

Together Propositions \ref{prop:locconst} and \ref{prop:arch} give the statement of Theorem \ref{thm:WAthm}.

\subsection*{Example}
Consider the Hassett--Pirutka--Tschinkel example over $\QQ$ where \[{F(x,y,z) = x^2 + y^2 + z^2 - 2(xy+xz+yz)}.\] 
	Let $X$ denote the variety defined by \eqref{eq:eqn} with this choice of $F$. The following two points of the form $(x,y,z;\mathbf{t})$ lie in $X(\QQ)$
	\[ P_1 =(1,1,1;1,1,1,1) \quad{} \text{ and } \quad{} P_2 = (1,1, -1; 1, 1 , 0, 0).\]
	We have 
	\begin{align*}
		\text{inv}_{\infty}\text{ev}_{\beta} (P_1) &= -1 = \text{inv}_{2}\text{ev}_{\beta} (P_1),\\
		\text{inv}_{\infty}\text{ev}_{\beta} (P_2) &= +1 = \text{inv}_{2}\text{ev}_{\beta} (P_2).
	\end{align*}
	Since the invariant maps are surjective at 2 and at the real place, weak approximation fails for $X$. Moreover, we have demonstrated that weak approximation can be obstructed at any of the places in $S$ (the finite set of primes in Theorem \ref{thm:WAthm}).

\end{document}